\documentclass[12pt,a4paper]{article}%
\usepackage[utf8]{inputenc}
\usepackage{hyperref}
\usepackage{amsmath}
\usepackage{amsfonts}
\usepackage{amssymb}
\usepackage{xcolor}
\usepackage{graphicx}%
\providecommand{\U}[1]{\protect\rule{.1in}{.1in}}
\newtheorem{theorem}{Theorem}

\newtheorem{lemma}[theorem]{Lemma}

\newtheorem{problem}[theorem]{Problem}
\newtheorem{proposition}[theorem]{Proposition}

\newenvironment{proof}[1][Proof]{\noindent\textbf{#1.} }{\ \hfill \rule{0.5em}{0.5em}\bigskip}
\graphicspath{{D:/Dropbox/Riste-Sedlar/MetricFaultTolerant/Slike/}}

\begin{document}

\title{Fault tolerance of metric basis can be expensive}
\author{Martin Knor$^{1}$, Jelena Sedlar$^{2,4}$, Riste \v{S}krekovski$^{3,4}$\\{\small $^{1}$ \textit{Slovak University of Technology, Bratislava, Slovakia
}}\\[0.1cm] {\small $^{2}$ \textit{University of Split, FGAG, Split, Croatia }}\\[0.1cm] {\small $^{3}$ \textit{University of Ljubljana, FMF, Ljubljana,
Slovenia }}\\[0.1cm] {\small $^{4}$ \textit{Faculty of Information Studies, Novo Mesto,
Slovenia }}}
\maketitle

\begin{abstract}
A set of vertices $S$ is a resolving set of a graph $G,$ if for every pair of
vertices $x$ and $y$ in $G$, there exists a vertex $s$ in $S$ such that $x$
and $y$ differ in distance to $s$. A smallest resolving set of $G$ is called a
metric basis. The metric dimension $\mathrm{dim}(G)$ is the cardinality of a
metric basis of $G$. The notion of a metric basis is applied to the problem of
placing sensors in a network, where the problem of sensor faults can arise.
The fault-tolerant metric dimension $\mathrm{ftdim}(G)$ is the cardinality of
a smallest resolving set $S$ such that $S\setminus\{s\}$ remains a resolving
set of $G$ for every $s\in S$. A natural question is how much more sensors
need to be used to achieve a fault-tolerant metric basis. It is known in
literature that there exists an upper bound on $\mathrm{ftdim}(G)$ which is
exponential in terms of $\mathrm{dim}(G),$ i.e. $\mathrm{ftdim}(G)\leq
\mathrm{dim}(G)(1+2\cdot5^{\mathrm{dim}(G)-1}).$ In this paper, we construct
graphs $G$ with $\mathrm{ftdim}(G)=\mathrm{dim}(G)+2^{\mathrm{dim}(G)-1}$ for
any value of $\mathrm{dim}(G)$, so the exponential upper bound is necessary.
We also extend these results to the $k$-metric dimension which is a
generalization of the fault-tolerant metric dimension. First, we establish a
similar exponential upper bound on $\mathrm{dim}_{k+1}(G)$ in terms of
$\mathrm{dim}_{k}(G),$ and then we show that there exists a graph for which
$\mathrm{dim}_{k+1}(G)$ is indeed exponential. For a possible further work, we
leave the gap between the bounds to be reduced.

\end{abstract}



\section{Introduction}

The notion of metric dimension is a classical notion in graph theory as it was
introduced in the 1970's. It first appeared in \cite{Harary1976} and
\cite{Slater1975}, with the motivation for its introduction being the problem
of locating robots moving through a network by the set of sensors placed in
some of the nodes of the network. This notion received much attention in the
scientific community, and there are many scientific papers covering the notion
from a broad variety of aspects. Some of the recent papers dealing with the
notion of metric dimension are \cite{Geneson,Hakanen, Mashkaria,Sedlar,Wu},
but this list is not nearly exhaustive. Many variants of the metric dimension
have also been introduced and extensively studied, such as the edge metric
dimension \cite{TratnikEdge}, the mixed metric dimension \cite{Kelm}, the
fault-tolerant metric dimension \cite{bibFtUpperBound, Mora2022, Vietz2019},
the $k$-metric dimension \cite{Corregidor, Estrada-Moreno2013, Klavzar}, the
weak $k$-metric dimension \cite{WeakDim} and so on. For a comprehensive survey
on the topic of metric dimension and its variants, one can see the two recent
surveys \cite{Kuziak, Tillquist}.

Throughout the paper, we assume that a graph $G$ is simple and connected,
unless explicitly stated otherwise. For a pair of vertices $x$ and $y$ in $G$,
their \emph{distance} $d_{G}(x,y)$ is defined as the length of a shortest path
connecting $x$ and $y$ in $G$. The \emph{distance} $d_{G}(x,S)$ between a
vertex $x$ and a set $S\subseteq V(G)$ is defined as $d_{G}(x,S)=\min
\{d_{G}(x,s):s\in S\}$.

A vertex $s\in V(G)$ \emph{distinguishes} a pair of vertices $x$ and $y\in
V(G)$ if $d_{G}(s,x)\not =d_{G}(s,y)$. If a vertex $s$ distinguishes $x$ and
$y$, it is sometimes said that it \emph{resolves} or \emph{locates} $x$ and
$y$. A set of vertices $S\subseteq V(G)$ is a \emph{resolving} set of $G$ if
for every pair of vertices $x$ and $y\in V(G)$, there exists a vertex $s\in S$
that distinguishes $x$ and $y$. A resolving set of $G$ is sometimes also
called a \emph{locating} set of $G$. A smallest resolving set of $G$ is called
a \emph{metric basis} of $G$, and the cardinality of a metric basis is called
the \emph{metric dimension} of $G$, which is denoted by $\mathrm{dim}(G)$.

Since the notion of metric dimension arises in the context of the problem of
placing sensors in a network in the physical world, it is clear that a sensor
in the physical world can malfunction. If that happens, the remaining
functional sensors may no longer distinguish all the pairs of nodes in a
network. Therefore, it is natural to pose the question of the resilience of
the set of sensors, i.e. how to choose a set of sensors so that it is
resilient to a sensor failure. This gives rise to a variant of metric
dimension known as fault-tolerant metric dimension. Namely, a set $S\subseteq
V(G)$ is a \emph{fault-tolerant resolving set} of $G$ if $S\setminus\{s\}$ is
a resolving set of $G$ for every $s\in S$. The cardinality of a smallest
fault-tolerant resolving set of $G$ is called the \emph{fault-tolerant metric
dimension} of $G$, and it is denoted by $\mathrm{ftdim}(G)$.

Note that every pair of vertices $x$ and $y\in V(G)$ is resolved by both $x$
and $y$. Therefore, it follows that every graph $G$ with at least two vertices
has a fault-tolerant resolving set. Hence, the fault-tolerant metric dimension
is well defined for every graph $G$ with $\left\vert V(G)\right\vert \geq2$.
Some of the relevant papers on the fault-tolerant metric dimension include
\cite{bibFtUpperBound, Mora2022,Vietz2019}.

The notion of the fault-tolerant resolving set, which is resilient to the
failure of precisely one sensor, can be generalized to the notion of a
resolving set that is resilient to the failure of $k$ sensors, for any integer
$k\geq2$. Namely, a set $S\subseteq V(G)$ is a $k$\emph{-resolving set} of $G$
if $S\setminus S^{\prime}$ is a resolving set of $G$ for every set $S^{\prime
}\subseteq S$ with $\left\vert S^{\prime}\right\vert =k-1$. The cardinality of
a smallest $k$-resolving set of $G$ is called the $k$-metric dimension of $G$,
and it is denoted by $\mathrm{dim}_{k}(G)$. Note that the fault-tolerant
metric dimension of $G$ is equal to the $k$-metric dimension of $G$ for $k=2$,
i.e. $\mathrm{ftdim}(G)=\mathrm{dim}_{2}(G)$. Further, a graph $G$ may not
have a $k$-resolving set for $k>2$. However, if a graph $G$ has a
$k$-resolving set for $k>2$, then it has a $k^{\prime}$-resolving set for
every $k^{\prime}\in\{1,\ldots,k\}$. The largest value of $k$ for which $G$
has a $k$-resolving set is denoted by $\kappa(G)$, and such a graph $G$ is
said to be $\kappa(G)$\emph{-dimensional}.

In this paper, we focus our attention on assuring the resistance of a sensor
set to sensor failure, which is modeled through the notions of fault-tolerant
resolving set and, more generally, a $k$-resolving set. A natural question
that arises in this context is how much this resistance to failure "costs",
i.e., how many more sensors we need to include in the sensor set $S$ to
achieve its resistance to the failure of $k$ sensors. We show that there are
broad classes of graphs where, in order to assure that a resolving set $S$ is
fault-tolerant, it is necessary to introduce to it at most $\left\vert
S\right\vert $ additional vertices.

Hernando, Mora, Slater and Wood \cite{bibFtUpperBound} established that in
general at most $2\left\vert S\right\vert \cdot5^{\left\vert S\right\vert -1}$
vertices must be added to a resolving set $S$ in order for it to become
fault-tolerant, i.e. resistant to the failure of precisely one sensor. Note
that this upper bound is exponential in $\left\vert S\right\vert $, which
implies that assuring the resistance of $S$ can be very costly even for only
one sensor failure. This means that $\mathrm{ftdim}(G)$ could be exponential
in terms of $\mathrm{dim}(G)$ in the worst-case scenario. In this paper, we
demonstrate that this indeed holds, i.e. that there exists a graph $G$ for
which $\mathrm{ftdim}(G)$ is exponential in terms of $\mathrm{dim}(G)$.
Moreover, we establish similar results for $\mathrm{dim}_{k}(G)$ for higher
values of $k$. More precisely, we show that $\mathrm{dim}_{k+1}(G)$ is bounded
from above by a function that is exponential in terms of $\mathrm{dim}_{k}(G)$
and that this bound is tight in the sense that there exists a graph $G$ for
which $\mathrm{dim}_{k+1}(G)$ is indeed exponential in terms of $\mathrm{dim}%
_{k}(G)$.

\section{Fault-tolerant metric dimension}

Let us consider the fault-tolerant dimension and the question how much larger
it can be compared to the metric dimension. We first show that for some of the
usual families of graphs, such as trees and complete multipartite graphs, it
holds that $\mathrm{ftdim}(G)$ is bounded from above by a linear function in
terms of $\mathrm{dim}(G).$ In \cite{bibFtUpperBound}, the exact value of the
fault-tolerant metric dimension is established for trees, and the following
proposition is a consequence.

\begin{proposition}
\label{Prop_trees}If $T$ is a tree, then%
\[
\mathrm{dim}(T)+1\leq\mathrm{ftdim}(T)\leq2\mathrm{dim}(T).
\]
Moreover, let $n\geq4$ and $1\leq d\leq\left\lfloor n/3\right\rfloor $. Then
there exists a tree $T$ on $n$ vertices with $\mathrm{ftdim}(T)-\mathrm{dim}%
(T)=d$.
\end{proposition}

\begin{proof}
If $T$ is a path then $\mathrm{dim}(T)=1$ and $\mathrm{ftdim}(T)=2$, so the
inequalities obviously hold. Let $T$ be a tree different from a path, and let
$a$ denote the number of leaves of $T$. Further, let $b$ (resp. $c$) be the
number of vertices that have degree greater than two and that are connected by
paths of degree two interior vertices to one or more (resp. to exactly one)
leaves. In \cite{Slater1975} it is established that $\mathrm{dim}(T)=a-b$ and
in \cite{bibFtUpperBound} that\ $\mathrm{ftdim}(T)=a-c.$

Now, observe that each of the $b-c$ branching vertices determines at least two
rays (paths with internal vertices of degree two terminating in a leaf), and
so%
\[
2(b-c)+c\leq a,
\]
which gives%
\[
\mathrm{ftdim}(T)=a-c\leq2(a-b)=2\mathrm{dim}(T).
\]
Further, consider a path on $d$ vertices. Attach to each vertex of the path at
least two vertices, and denote by $T$ the resulting tree on $n$ vertices.
Since $a=n-d$, $b=d$ and $c=0$, we get $\mathrm{ftdim}(T)=n-d$ and
$\mathrm{dim}(T)=n-2d$.
\end{proof}

Next, we show that the same linear upper bound holds for complete multipartite graphs.

\begin{proposition}
\label{Prop_multipartite}If $G$ is a complete multipartite graph, then%
\[
\mathrm{dim}(G)+1\leq\mathrm{ftdim}(G)\leq2\mathrm{dim}(G).
\]
Moreover, let $n\geq2$ and let $1\leq d\leq\left\lfloor n/2\right\rfloor $.
Then there exists a complete multipartite graph $G$ on $n$ vertices with
$\mathrm{ftdim}(G)-\mathrm{dim}(G)=d$.
\end{proposition}

\begin{proof}
Let $G$ be a complete multipartite graph on $n$ vertices with $p$ partite
sets, out of which $q$ consist of single vertices. We will first show that%
\[
\mathrm{dim}(G)=\left\{
\begin{array}
[c]{ll}%
n-p & \text{if }q=0,\\
n-p+q-1 & \text{otherwise.}%
\end{array}
\right.
\]
To see this, let $x$ and $y$ be two distinct vertices from one partite set.
Then no vertices other than $x$ and $y$ distinguish them. Hence, every
resolving set $S$ can miss at most one vertex from each partite set. For
$q=0,$ it is easily verified that a set of vertices $S$ which misses precisely
one vertex from each partite set is a resolving set of $G,$ which implies
$\mathrm{dim}(G)=n-p.$ Otherwise, for $q\not =0,$ let $x$ and $y$ are two
vertices from distinct partite sets, both of which consist of a single vertex.
Notice that $x$ and $y$ are resolved only by themselves. This implies that a
resolving set $S$ can miss only one vertex from single vertex partite sets. It
is easily verified that a set $S\subseteq V(G)$ which misses precisely one
vertex from each partite set with at least two vertices and a vertex from only
one single vertex partite set is a resolving set of $G,$ from which we
conclude $\mathrm{dim}(G)=n-p+q-1.$

Let us next show that
\[
\mathrm{ftdim}(G)=\left\{
\begin{array}
[c]{ll}%
n-1 & \text{if }q=1,\\
n & \text{otherwise.}%
\end{array}
\right.
\]
In order to do that, let $S\subseteq V(G)$ be a set of vertices distinct from
$V(G),$ so there exists a vertex $x\in V(G)\backslash S.$ If $x$ belongs to a
partite set with at least two elements, then there exists a vertex $y$ from
the same partite set as $x.$ In this case the set $S\backslash\{y\}$ does not
resolve $G,$ so $S$ is not a fault-tolerant resolving set of $G.$ We conclude
that a fault-tolerant resolving set must contain all vertices from partite
sets with at least two elements. Hence, for $q=0,$ we have $\mathrm{ftdim}%
(G)=n.$ Next, for $q=1,$ let $x$ be the only vertex contained in a single
vertex partite set. It is easily verified that the set $S\backslash\{x\}$ is a
fault-tolerant resolving set, so $\mathrm{ftdim}(G)=n-1.$ Finally, for
$q\geq2$, let $x$ and $y$ be two vertices, such that each of them belongs to a
single partite set of $G.$ Assume that $x$ is not contained in $S,$ then
$S\backslash\{y\}$ is not a resolving set of $G,$ so we obtain $\mathrm{ftdim}%
(G)=n$..

Now that we have the exact value of both the metric dimension and the
fault-tolerant metric dimension, let us prove the claim of the proposition.
The inequalities are obvious from the fact that at most one partite set which
is singleton can miss in a resolving set, and from non-singleton partite sets
can as well miss only one vertex in a resolving set. The second part is
achieved by a complete multipartite graph which has exactly $d$ partite sets,
all of which have order at least $2$.
\end{proof}

Propositions \ref{Prop_trees} and \ref{Prop_multipartite} give us broad
classes of graphs $G$ for which $\mathrm{ftdim}(G)$ is bounded above by a
linear function in terms of $\mathrm{dim}(G)$. But the theoretical bound for
$\mathrm{ftdim}(G)$ in terms of $\mathrm{dim}(G)$ is given in the following
theorem established in \cite{bibFtUpperBound}.

\begin{theorem}
\label{Tm_spanci}Fault-tolerant metric dimension is bounded by a function of
the metric dimension (independent of the graph). In particular, for every
graph $G$ we have
\[
\mathrm{ftdim}(G)\leq\mathrm{dim}(G)(1+2\cdot5^{\mathrm{dim}(G)-1}).
\]

\end{theorem}

The above theorem implies that in the worst case scenario the fault-tolerant
dimension of a graph $G$ might be exponential in terms of the metric dimension
of $G$. Yet, this is only an upper bound on $\mathrm{ftdim}(G)$ which does not
have to be tight. To our knowledge, there is no result in literature which
shows that the upper bound on $\mathrm{ftdim}(G)$ needs to be exponential in
terms of $\mathrm{dim}(G).$ We will show that a graph $G$ for which
$\mathrm{ftdim}(G)$ is exponential indeed does exist, and that this holds for
any value of $\mathrm{dim}(G)$. First, let us introduce a particular kind of
graphs which will be of use to us in proving this claim.

\begin{figure}[h]
\begin{center}
\includegraphics[scale=0.6]{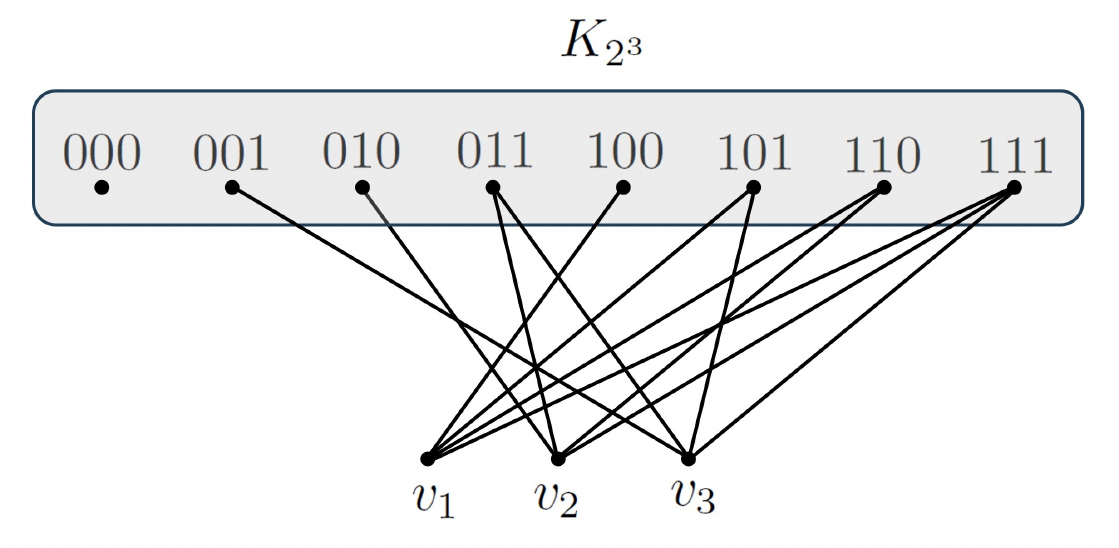}
\end{center}
\caption{The figure shows the graph $M_{t}$ from the proof of Theorem
\ref{Tm_ft} for $t=3$. The graph $M_{t}$ consists of a complete graph
$K_{2^{t}}$ whose vertices are labeled by binary codes of length $t$ and
additional vertices $v_{i}$ for $1\leq i\leq t$. A vertex $v_{i}$ is connected
to all vertices of $K_{2^{t}}$ which have $1$ on the $i$-th coordinate.}%
\label{Fig01}%
\end{figure}

Let $t\geq1$ be an integer and let $n=t+2^{t}$. We construct the graph $M_{t}$
on $n=t+2^{t}$ vertices, illustrated by Figure~{\ref{Fig01}}, as follows. Let
$V_{1}=\{u_{0},\ldots,u_{2^{t}-1}\}$ be a set of $2^{t}$ vertices and let
$V_{2}=\{v_{1},\ldots,v_{t}\}$ be a set of $t$ vertices. We identify a vertex
$u_{j}$ of $V_{1}$ with the binary code for $j$, i.e., we assume that $V_{1}$
consists of all binary codes of length $t$. For example, if $t=3$, then
$V_{1}=\{000,001,010,011,100,101,110,111\}$. Graph $M_{t}$ is defined to have
the set of vertices $V(M_{t})=V_{1}\cup V_{2}$. Further, each pair of vertices
from $V_{1}$ is connected by an edge in $M_{t}$, so vertices of $V_{1}$ induce
a complete subgraph $K_{2^{t}}$ in $M_{t}$. Vertex $v_{i}$ from $V_{2}$ is
adjacent in $M_{t}$ to all vertices of $V_{1}$ which have 1 on the $i$-th coordinate.

\begin{theorem}
\label{Tm_ft} For every $t\geq1,$ it holds that $\mathrm{dim}(M_{t})=t$ and
$\mathrm{ftdim}(M_{t})=t+2^{t-1}$.
\end{theorem}

\begin{proof}
Let us first consider the small values of $t.$ For $t=1$ we have $M_{t}=P_{3}%
$, so the claim holds. For $t=2$, the graph $M_{t}$ has $6$ vertices, so the
claim is easily verified. Hence, we may assume $t\geq3$.

\medskip

\noindent\textbf{Claim A.} \emph{The set }$S=V_{2}\subseteq V(M_{t})$\emph{ is
a resolving set of }$M_{t}.\smallskip$

\noindent Let $x,y\in V(M_{t})$. If at least one of the vertices $x$ and $y$
belongs to $V_{2}$, say $x$, then $x$ and $y$ are resolved by $x\in S$. Assume
that both $x$ and $y$ belong to $V_{1}$. Since $x$ and $y$ are two distinct
vertices, they differ in at least one coordinate. Say, $x$ and $y$ differ in
$i$-th coordinate, where $x$ has $1$ and $y$ has $0$ on the $i$-th coordinate.
Then $x$ is adjacent to $v_{i}\in S$, and $y$ is not. This implies $x$ and $y$
are distinguished by $v_{i}\in S$. We conclude that $x$ and $y$ are resolved
by $S$ in all possible cases, so $S$ is a resolving set of $M_{t}$.

\medskip

\noindent\textbf{Claim B.} \emph{If }$S\subseteq V(M_{t})$\emph{ is a
resolving set in }$M_{t}$\emph{ and }$\left\vert S\cap V_{2}\right\vert
=t-1$,\emph{ then }$\left\vert S\right\vert \geq t-1+2^{t-1}.\smallskip$

\noindent From $\left\vert S\cap V_{2}\right\vert =t-1$, we conclude that $S$
does not contain precisely one vertex of $V_{2}$, say $v_{i}$. Let $x$ and $y$
be a pair of vertices from $V_{1}$ which differ only in the $i$-th coordinate,
say $x$ has $0$ on the $i$-th coordinate and $y$ has $1$. Then $x$ and $y$ are
not distinguished by any vertex of $V_{2}\setminus\{u_{i}\}$. Also, $x$ and
$y$ are not distinguished by any vertex of $V_{1}\setminus\{x,y\}$. This
implies that the set $S$ must contain one vertex from $\{x,y\}$ for each such
pair. Notice that for each possible combination of $0$'s and $1$'s on the part
of the code distinct from the $i$-th coordinate, there is a pair $x$ and $y$
which can be constructed by adding $0$ (resp. $1$) to the $i$-th coordinate.
Hence, there are $2^{t-1}$ such pairs $\{x,y\}$ and they are all pairwise
vertex disjoint. Since $S$ must contain at least one vertex from each such
pair, we conclude that $\left\vert S\right\vert =\left\vert S\cap
V_{1}\right\vert +\left\vert S\cap V_{2}\right\vert \geq t-1+2^{t-1}$.

\medskip

\noindent\textbf{Claim C.} \emph{If }$S\subseteq V(M_{t})$\emph{ is a
resolving set in }$M_{t}$\emph{ and }$\left\vert S\cap V_{2}\right\vert \leq
t-2$,\emph{ then }$\left\vert S\right\vert \geq t-1+2^{t-1}.\smallskip$

\noindent The assumption $\left\vert S\cap V_{2}\right\vert \leq t-2$ implies
that the set $S$ does not contain two vertices of $V_{2}$, say $v_{i}$ and
$v_{j}$. Let $\{x,y,z,w\}\subseteq V_{1}$ be a quadruple of vertices which
coincide on all coordinates except possibly on the $i$-th and/or $j$-th
coordinate. For each pair of vertices from such a quadruple to be
distinguished by $S$, the set $S$ must contain three out of four vertices of
the quadruple. Since there are $2^{t-2}$ such quadruples and they are all
pairwise vertex disjoint, we conclude that $S$ must contain $3\cdot2^{t-2}$
vertices from $V_{1}$. Hence, $\left\vert S\right\vert \geq3\cdot2^{t-2}\geq
t-1+2^{t-1}$ for $t\geq3$, and we are finished.

\medskip

Notice that $t-1+2^{t-1}>t$, hence Claims~B and C imply that the set $S$ from
Claim~A is the smallest resolving set of $M_{t}$. Hence, $\mathrm{dim}%
(M_{t})=t$. It remains to establish that $\mathrm{ftdim}(M_{t})=t+2^{t-1}$.
For that purpose, we need the following claims.

\medskip

\noindent\textbf{Claim D.} \emph{There exists a fault-tolerant resolving set
}$S$\emph{ in }$M_{t}$\emph{ with }$\left\vert S\right\vert =t+2^{t-1}%
.\smallskip$

\noindent Let $S$ be a set which consists of all vertices from $V_{2}$ and all
those vertices from $V_{1}$ which have an even number of $1$'s in their binary
code. Obviously, $\left\vert S\right\vert =t+2^{t-1}$, so it remains to
establish that $S$ is a fault-tolerant resolving set in $M_{t}$. That is, we
have to show that $S\setminus\{w\}$ is a resolving set of $M_{t}$ for every
$w\in S$. If $w\in V_{1}$, then $V_{2}\subseteq S\setminus\{w\}$, so Claim~A
implies that $S\setminus\{w\}$ is a resolving set in $M_{t}$. It remains to
consider the case when $w\in V_{2}$.

Let $x,y$ be a pair of vertices in $M_{t}$. If $\{x,y\}\subseteq V_{2}$, then
at least one of $x$ and $y$ belongs to $S\setminus\{w\}$, say $x$. Hence, $x$
and $y$ are distinguished by $x\in S\setminus\{w\}$. If $\{x,y\}\subseteq
V_{1}$, then we distinguish two cases. In the case when at least one of $x$
and $y$ has an even number of $1$'s, say $x$, then $x$ belongs to
$S\setminus\{w\}$, so $x$ and $y$ are distinguished by $x\in S\setminus\{w\}$.
In the case when both $x$ and $y$ have an odd number of $1$'s, then $x$ and
$y$ differ in at least two coordinates, say coordinates $i$ and $j$. Since
$S\setminus\{w\}$ contains all vertices of $V_{2}$ except $w$, this implies
that one of $v_{i}$ and $v_{j}$ must be included in $S\setminus\{w\}$, say
$v_{i}\in S\setminus\{w\}$. Then $x$ and $y$ are distinguished by $v_{i}\in
S\setminus\{w\}$. Finally, assume that $x\in V_{1}$ and $y\in V_{2}$. Then the
vertex of $V_{1}$ with coordinates $0$ only has distance at most $1$ to $x$,
while its distance to $y$ is $2$. Since this vertex is in $S\setminus\{w\}$
(recall that $w\in V_{2}$), the set $S\setminus\{w\}$ distinguishes $x$ and
$y$. So, we have established that $x$ and $y$ are distinguished by
$S\setminus\{w\}$ in all possible cases, which implies $S$ is a $2$-resolving
set as claimed.

\medskip

In order to establish that $\mathrm{ftdim}(M_{t})=t+2^{t-1}$, it remains to
show that the resolving set from Claim~D is a smallest fault-tolerant
resolving set.

\medskip

\noindent\textbf{Claim E.} \emph{If }$S\subseteq V(M_{t})$\emph{ is a
fault-tolerant resolving set in }$G$,\emph{ then }$\left\vert S\right\vert
\geq t+2^{t-1}.\smallskip$

\noindent Notice that $S^{\prime}=S\setminus\{w\}$ must be a resolving set for
every $w\in S$. If $S$ contains a vertex of $V_{2}$, let $w\in S\cap V_{2}$.
Denote $S^{\prime}=S\setminus\{w\}$ and notice that Claims B and C imply that
$\left\vert S^{\prime}\right\vert \geq t-1+2^{t-1}$, so $\left\vert
S\right\vert =\left\vert S^{\prime}\right\vert +1\geq t+2^{t-1}$, as claimed.
If $S$ does not contain a vertex of $V_{2}$, then let $w$ be any vertex of
$S$. Notice that $t\geq3$ together with Claim C imply that $\left\vert
S^{\prime}\right\vert \geq t-1+2^{t-1}$, so again we have $\left\vert
S\right\vert \geq t+2^{t-1}$ as claimed.

\medskip

Notice that Claims D and E imply that $\mathrm{ftdim}(M_{t})=t+2^{t-1}$, which
concludes the proof.
\end{proof}

\section{$k$-metric dimension}

Since the notion of the fault-tolerant metric dimension is generalized by the
$k$-metric dimension, it is natural to ask how the $k$-metric dimension
$\mathrm{dim}_{k+1}(G)$ behaves in terms of $\mathrm{dim}_{k}(G)$, for
$k\geq2$, provided that $G$ is a graph with $\kappa(G)\geq k+1$. First, let us
briefly comment the lower bound on $\mathrm{dim}_{k+1}(G)$. The following
result is established in \cite{Estrada-Moreno2013}.

\begin{proposition}
\label{Prop_Yero}Let $k\geq2$ and let $G$ be a graph with $\kappa(G)\geq k+1$.
Then $\mathrm{dim}_{k+1}(G)\geq\mathrm{dim}_{k}(G)+1$.
\end{proposition}

Let $k\geq2$ and $n\geq k+2$ be a pair of integers. It is easily seen that for
the path $P_{n}$ it holds that
\[
\mathrm{dim}_{k}(P_{n})=\left\{
\begin{array}
[c]{ll}%
k & \text{for }k\leq2,\\
k+1 & \text{otherwise.}%
\end{array}
\right.
\]
Hence, the path $P_{n}$ achieves the lower bound from Proposition
\ref{Prop_Yero} for every $k\not =2$.

Let us next consider the upper bound on $\mathrm{dim}_{k+1}(G)$ in terms of
$\mathrm{dim}_{k}(G)$. We will use the similar approach by which the bound
from Theorem \ref{Tm_spanci} is obtained in \cite{bibFtUpperBound}, so we
start with the following lemma.

\begin{lemma}
\label{Lemma_Jelena}Let $k\geq2$ be an integer, $G$ a graph with
$\kappa(G)\geq k+1$ and $S\subseteq V(G)$ a $k$-resolving set of $G$. Let
$x,y\in V(G)$ be a pair of vertices in $G$ and $\{s_{1},\ldots,s_{k}%
\}\subseteq S$ a set of vertices which distinguish $x$ and $y$. Then, there
exists a vertex $s\in V(G)\setminus\{s_{1},\ldots,s_{k}\}$ such that
$d_{G}(s,S)\leq2$ and $s$ distinguishes $x$ and $y$.
\end{lemma}

\begin{proof}
Since $G$ is a $\kappa(G)$-dimensional graph for $\kappa(G)\geq k+1$, there
must exist a vertex $s\in V(G)\setminus\{s_{1},\ldots,s_{k}\}$ which
distinguishes $x$ and $y$. We have to show that for at least one such vertex
$s$ it holds that $d_{G}(s,S)\leq2$. By way of contradiction, assume that for
every vertex $s\in V(G)\setminus\{s_{1},\ldots,s_{k}\}$ which distinguishes
$x$ and $y$ it holds that $d_{G}(s,S)\geq3$. Our assumption implies that
$s_{1},\ldots,s_{k}$ are the only vertices of $S$ which distinguish $x$ and
$y$. Let $s\in V(G)\setminus S$ be a vertex which distinguishes $x$ and $y$
such that $d_{G}(s,S)$ is minimum possible. Notice that $x$ and $y$ are
distinguished by each of the vertices $x$ and $y$, so we conclude that
$x,y\not \in S\setminus\{s_{1},\ldots,s_{k}\}$. Hence, for each of $x$ and $y$
it holds that it belongs either to $\{s_{1},\ldots,s_{k}\}$ or to
$V(G)\setminus S$. We distinguish three cases.

\medskip

\noindent\textbf{Case 1:} $x$\emph{ belongs to }$\{s_{1},\ldots,s_{k}\}$\emph{
and }$y$\emph{ to }$V(G)\setminus S$. Let $P$ be a shortest path which
connects $x$ and $y$. Assume first that $P$ is of an odd length. Then each
vertex of $P$ distinguishes $x$ and $y$. Since $x$ belongs to $S$ and $y$ to
$V(G)\setminus S$, there must exist a vertex $s^{\prime}$ of $P$ which does
not belong to $S$ such that the distance $d_{G}(x,s^{\prime})$ is the smallest
possible. Notice that it can happen $s^{\prime}=y$. Obviously, it holds that
$d_{G}(s^{\prime},S)=1$. Hence, we have found a vertex $s^{\prime}\in
V(G)\setminus S$ with $d_{G}(s^{\prime},S)=1$ which distinguishes $x$ and $y$,
a contradiction with our assumption that $d_{G}(s,S)\geq3$ for every such
vertex $s$.

Assume next that $P$ is of an even length. Then, the length of $P$ is at least
two, and there exists precisely one vertex $p$ on $P$ which does not
distinguish $x$ and $y$. Notice that the vertex $p$ is the middle vertex of
$P$, hence $p\not \in \{x,y\}$. Again, let $s^{\prime}$ be a vertex of
$V(P)\setminus S$ with $d_{G}(s^{\prime},x)$ minimum possible. Since
$x\in\{s_{1},\ldots,s_{k}\}\subseteq S$ and $y\in V(G)\setminus S$, such a
vertex $s^{\prime}$ must exist. Also, it obviously holds that $d_{G}%
(s^{\prime},S)=1$, as otherwise there would exist a vertex of $V(P)\setminus
S$ closer to $x$ than $s^{\prime}$.

Now, if $s^{\prime}\not =p$, we have again found a vertex $s^{\prime}\in
V(G)\setminus S$ which distinguishes $x$ and $y$, and which satisfies
$d_{G}(s^{\prime},S)=1$, a contradiction. Otherwise, if $s^{\prime}=p$, let
$p^{\prime}$ be the neighbor of $s^{\prime}$ on $P$ with $d_{G}(x,p^{\prime
})=d_{G}(x,p)+1$. Denote by $P^{\prime}$ the subpath of $P$ connecting
$p^{\prime}$ and $y$. Notice that it may happen $p^{\prime}=y$, in which case
the path $P^{\prime}$ consists of a single vertex. Let $s^{\prime\prime}$ be
the vertex of $P^{\prime}$ which does not belong to $S$ chosen so that the
distance $d_{G}(s^{\prime\prime},S)$ is the minimum possible. Since
$y\not \in S$, such a vertex $s^{\prime\prime}$ must exist on $P^{\prime}$. If
$s^{\prime\prime}=p^{\prime}$, this means that $s^{\prime\prime}$ is the
neighbor of $s^{\prime}$ on $P$. Observe that $d_{G}(s^{\prime},S)=1$ implies
that $d_{G}(s^{\prime\prime},S)\leq2$. Since $s^{\prime\prime}\not =p$ and $p$
is the only vertex of $P$ which does not distinguish $x$ and $y$, we conclude
that $s^{\prime\prime}$ does distinguish $x$ and $y$. Therefore, we have again
found a vertex $s^{\prime\prime}\in V(G)\setminus S$ with $d(s^{\prime\prime
},S)\leq2$ which distinguishes $x$ and $y$, a contradiction. On the other
hand, if $s^{\prime\prime}\not =p^{\prime},$ this implies that the neighbor of
$s^{\prime\prime}$ on $P^{\prime}$ belongs to $S,$ as otherwise the distance
$d_{G}(s^{\prime\prime},S)$ would not be the minimum possible. So,
$s^{\prime\prime}$ is the first vertex of $P^{\prime}$ after $s^{\prime}=p$ in
the direction of $y$ which does not belong to $S,$ and such a vertex
$s^{\prime\prime}$ must exist since $y\not \in S.$ So, we have $d(s^{\prime
\prime},S)=1$ and $s^{\prime\prime}\not =p$ implies that $s^{\prime\prime}$
distinguishes $x$ and $y,$ a contradiction.

\medskip

\noindent\textbf{Case 2:} \emph{both }$x$\emph{ and }$y$\emph{ belong to
}$\{s_{1},\ldots,s_{k}\}$. Recall that $s\in V(G)\setminus S$ is a vertex
which distinguishes $x$ and $y$ such that $d_{G}(s,S)$ is minimum possible.
Let $P_{x}$ (resp. $P_{y}$) be a shortest path which connects $x$ and $s$
(resp. $y$ and $s$). Since $x$ and $y$ are distinguished by $s$, it holds that
$d_{G}(x,s)\not =d_{G}(y,s)$. Without loss of generality, we may assume that
$d_{G}(x,s)<d_{G}(y,s)$. We show that every vertex of $P_{x}$ must distinguish
$x$ and $y$. Assume to the contrary, that there exists a vertex $q\in
V(P_{x})$ such that $d_{G}(x,q)=d_{G}(y,q)$. Then we have $d_{G}(y,s)\leq
d_{G}(y,q)+d_{G}(q,s)=d_{G}(x,q)+d_{G}(q,s)=d_{G}(x,s)$, a contradiction with
$d_{G}(x,s)<d_{G}(y,s)$. Now, since $x$ and $s$ are the two end-vertices of
$P_{x}$, where $x\in S$ and $s\not \in S$, there must exist a vertex
$s^{\prime}\in V(P_{x})\setminus V(S)$ with $d_{G}(s^{\prime},x)$ being
minimum possible. For such a vertex $s^{\prime}$ it holds that $d_{G}%
(s^{\prime},S)=1$ and $s^{\prime}$ distinguishes $x$ and $y$, a contradiction.

\medskip

\noindent\textbf{Case 3:} \emph{both }$x$\emph{ and }$y$\emph{ belong to
}$V(G)\setminus S$. Let $d_{x}=\min\{d_{G}(x,s_{i}):1\leq i\leq k\}$ and
$d_{y}=\min\{d_{G}(y,s_{i}):1\leq i\leq k\}$. Without loss of generality we
may assume $d_{x}\leq d_{y}$. Let $s_{i}\in\{s_{1},\ldots,s_{k}\}$ be a vertex
with $d_{G}(s_{i},x)=d_{x}$. Let $P_{x}$ (resp. $P_{y}$) be a shortest path
connecting $s$ and $x$ (resp. $s$ and $y$). Since $s_{i}$ distinguishes $x$
and $y$, it holds that $d_{G}(x,s_{i})\not =d_{G}(y,s_{i})$. From $d_{x}\leq
d_{y}$ and $d_{G}(x,s_{i})=d_{x}$ we infer that $d_{G}(x,s_{i})<d_{G}%
(y,s_{i})$. Let us show that every vertex of $P_{x}$ must distinguish $x$ and
$y$. To see this, assume to the contrary that there exists a vertex $p\in
V(P_{x})$ with $d_{G}(x,p)=d_{G}(y,p)$. Then we have $d_{G}(y,s_{i})\leq
d_{G}(y,p)+d_{G}(p,s_{i})=d_{G}(x,p)+d_{G}(p,s_{i})=d_{G}(x,s_{i})$ which is a
contradiction with $d_{G}(x,s_{i})<d_{G}(y,s_{i})$. Hence, we have established
that every vertex of $P_{x}$ distinguishes $x$ and $y$. Recall that the
end-vertices of $P_{x}$ are $x$ and $s_{i}$, where $x\not \in S$ and $s_{i}\in
S$. This implies that there must exist a vertex $s^{\prime}\in V(P_{x}%
)\setminus S$ with $d_{G}(s^{\prime},s_{i})$ minimum possible. Obviously,
$d_{G}(s^{\prime},S)=1$ and $s^{\prime}$ distinguishes $x$ and $y$, a contradiction.
\end{proof}

We also need the following lemma from \cite{bibFtUpperBound}.

\begin{lemma}
\label{Lemma_spanci}Let $S$ be a resolving set in a graph $G$. Then for each
vertex $v\in S$, the number of vertices of $G$ at distance at most $d$ from
$v$ is at most $1+d(2d+1)^{\left\vert S\right\vert -1}$.
\end{lemma}

We are now in a position to establish the following theorem which gives an
upper bound on $\mathrm{dim}_{k+1}(G)$ in terms of $\mathrm{dim}_{k}(G)$.

\begin{theorem}
\label{Tm_boundK}Let $k\geq2$ be an integer and $G$ a graph with
$\kappa(G)\geq k+1$. The $(k+1)$-metric dimension of $G$ is bounded by a
function of the $k$-metric dimension of $G$. In particular, $\mathrm{dim}%
_{k+1}(G)\leq\mathrm{dim}_{k}(G)(1+2\cdot5^{\mathrm{dim}_{k}(G)-1})$.
\end{theorem}

\begin{proof}
Let $S$ be a $k$-resolving set of $G$. It is sufficient to show that there
exists a $(k+1)$-resolving set $S^{\prime}$ in $G$ with $\left\vert S^{\prime
}\right\vert \leq\left\vert S\right\vert (1+2\cdot5^{\left\vert S\right\vert
-1})$. We define $S^{\prime}$ to consist of all vertices of $G$ whose distance
to $S$ is at most two, i.e. $S^{\prime}=\{v\in V(G):d_{G}(v,S)\leq2\}$.

First, let us establish that $S^{\prime}$ is a $(k+1)$-resolving set. Let $x$
and $y$ be a pair of vertices of $G$, and let $\{s_{1},\ldots,s_{k}\}\subseteq
S$ be a set of $k$ vertices which distinguish $x$ and $y$. Lemma
\ref{Lemma_Jelena} implies that there exists a vertex $s\in V(G)\setminus
\{s_{1},\ldots,s_{k}\}$ such that $d(s,S)\leq2$ and $s$ distinguishes $x$ and
$y$. By the definition of $S^{\prime}$ we know that $s\in S^{\prime}$. We
conclude that $S^{\prime}$ contains $k+1$ vertices $s_{1},\ldots,s_{k},s$
which distinguish $x$ and $y$, so $S^{\prime}$ is a $(k+1)$-resolving set of
$G$.

Next, let us establish the desired upper bound on $\left\vert S^{\prime
}\right\vert $. Since $S$ is a $k$-resolving set of $G$, and $k\ge2$, $S$ is
also a resolving set of $G$. Hence we can apply Lemma \ref{Lemma_spanci} to it
with $d=2$. This yields that for each vertex $v$ of $S$, the graph $G$
contains at most $1+2\cdot5^{\left\vert S\right\vert -1}$ vertices at distance
at most $2$ from $v$. We obtain that $\left\vert S^{\prime}\right\vert
\leq\left\vert S\right\vert (1+2\cdot5^{\left\vert S\right\vert -1})$ and we
are done.
\end{proof}

\begin{figure}[h]
\begin{center}
\includegraphics[scale=0.6]{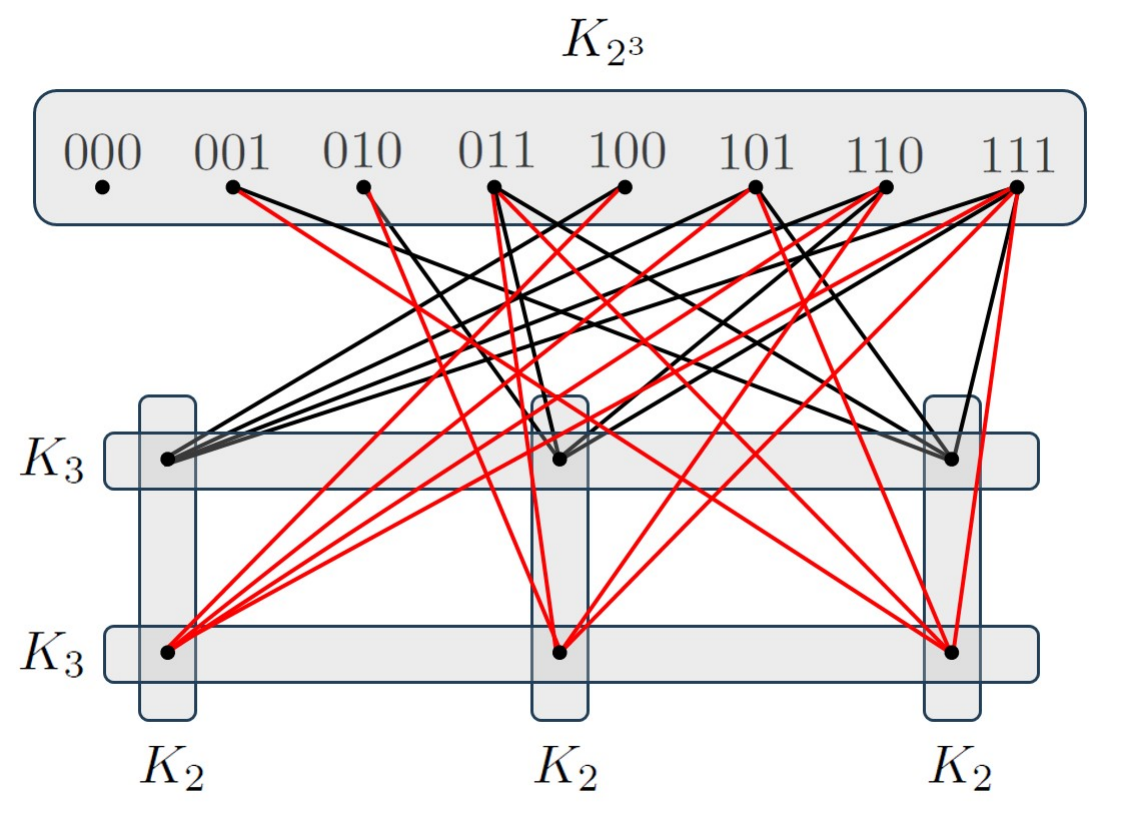}
\end{center}
\caption{The graph $M_{t,k}$ from the proof of Theorem \ref{Tm_Knor2} for
$t=3$ and $k=2$. The graph $M_{t,k}$ consists of a complete graph $K_{2^{t}}$
whose vertices are labeled by binary codes of length $t$ and vertex disjoint
graph $K_{t}\square K_{k}$ with vertices $v_{i,j}$ for $1\leq i\leq t$ and
$1\leq j\leq k$. A vertex $v_{i,j}$ of $K_{t}\square K_{k}$ is connected to
all vertices of $K_{2^{t}}$ which have $1$ on the $i$-th coordinate.}%
\label{Fig02}%
\end{figure}

Similarly as with the bound of Theorem \ref{Tm_spanci}, the upper bound of
Theorem \ref{Tm_boundK} is a theoretical bound and the theorem gives no answer
to the question how tight the bound is. We will show that there indeed exists
a graph $G$ such that the $(k+1)$-metric dimension of $G$ is exponential in
terms of the $k$-metric dimension of $G$. In order to do so, we construct
similar graphs as in the previous section.

Let $t\geq1$ and $k\geq2$ be two integers. The graph $M_{t,k}$ on $n=kt+2^{t}$
vertices, illustrated by Figure \ref{Fig02}, is constructed as follows. Let
$V_{1}$ be a set of $2^{t}$ vertices denoted by sequences of $0$'s and $1$'s
of length $t$. Construct a complete graph on $V_{1}$. Now take the Cartesian
product of $K_{t}\square K_{k}$ with vertices $v_{i,j}$, where $1\leq i\leq t$
and $1\leq j\leq k$, and denote by $V_{2}$ the vertices of this Cartesian
product. Further, join $v_{i,j}$ with exactly those vertices of $V_{1}$ which
have $1$ on $i$-th position, for $1\leq i\leq t$. Finally, the resulting graph
is denoted by $M_{t,k}.$ In the next theorem we prove that the graphs
$M_{t,k}$ have the desired property.

\begin{theorem}
\label{Tm_Knor2}For every $k\geq2$ and every big enough $t,$ it holds that
$\mathrm{dim}_{k}(M_{t,k})=kt$ and $\mathrm{dim}_{k+1}(M_{t,k})\geq2^{t-1}$.
\end{theorem}

\begin{proof}
We will prove the theorem through several claims.

\medskip

\noindent\textbf{Claim A.} \emph{The set }$S=V_{2}\subseteq V(M_{t,k})$\emph{
is a }$k$\emph{-resolving set of }$M_{t,k}.\smallskip$

\noindent Let $S^{\prime}$ be a set obtained from $S$ by deleting $k-1$
vertices. It is sufficient to prove that $S^{\prime}$ is a resolving set of
$M_{t,k}$. Let $x,y\in V(M_{t,k})$ be a pair of vertices of $M_{t,k}$, we have
to show that $x$ and $y$ are resolved by $S^{\prime}$.

Assume first that both $x$ and $y$ belong to $V_{1}$. Observe that there are
$k$ copies of $K_{t}$ in the product $K_{t}\square K_{k}$. So in order to
remove from $S$ a vertex from every copy of $K_{t}$, one would need to remove
at least $k$ vertices. Since $S^{\prime}$ is obtained from $S$ by removing
$k-1$ vertices, we conclude that $S^{\prime}$ contains all the vertices of at
least one copy of $K_{t}$. Since vertices of each copy of $K_{t}$ distinguish
all pairs of vertices from $V_{1}$, we conclude that $S^{\prime}$
distinguishes $x$ and $y$.

Assume next that both $x$ and $y$ belong to $V_{2}$. Notice that a pair of
vertices $x,y\in V(M_{t,k})$ is always distinguished by both $x$ and $y$.
Hence, if at least one of the vertices $x$ and $y$ belong to $S^{\prime}$, we
are done. So, let us assume that $x,y\in V_{2}\setminus S^{\prime}$. In order
to remove all vertices of one copy of $K_{k}$ from $S$, one would need to
remove $k$ vertices from it. Since $S^{\prime}$ is obtained from $S$ by
removing $k-1$ vertices, we conclude that $S^{\prime}$ contains at least one
vertex from each copy of $K_{k}$. Similarly, to remove all vertices of one
copy of $K_{t}$ from $S$, one needs to remove $t$ vertices from it. Since $t$
is big enough, we may assume $t>k$, so we conclude that $S^{\prime}$ contains
at least one vertex also from each copy of $K_{t}$.

Now, if $x$ and $y$ belong to two distinct copies of $K_{k}$, recall that
$S^{\prime}$ contains all the vertices of at least one copy of $K_{t}$. We
chose for $s$ the vertex of this copy of $K_{t}$ which belongs to the same
copy of $K_{k}$ as $x$. Since the entire copy of $K_{t}$ belongs to
$S^{\prime}$, it follows that $s\in S^{\prime}$. Recall that we assumed
$x,y\not \in S^{\prime}$, so neither $x$ nor $y$ belong to this copy of
$K_{t}$. Since $x$ and $s$ belong to the same copy of $K_{k}$, it follows that
$x$ and $s$ are connected by an edge in $K_{t}\square K_{k}$. On the other
hand, $y$ and $s$ do not belong to the same copy of $K_{t}$ nor to the same
copy of $K_{k}$ (since $y$ belongs to a distinct copy of $K_{k}$ than $x$).
Hence, $y$ and $s$ are not connected by an edge in $K_{t}\square K_{k}$. We
conclude that $s\in S^{\prime}$ distinguishes $x$ and $y$.

Next, if $x$ and $y$ belong to the same copy of $K_{k}$, then they must belong
to two distinct copies of $K_{t}$. Let $s$ be a vertex of $S^{\prime}$ which
belongs to the same copy of $K_{t}$ as $x$. Since $x\not \in S^{\prime}$, it
follows that $x$ and $s$ are connected by an edge in $K_{t}\times K_{k}$. On
the other hand, since $s$ belongs to the same copy of $K_{t}$ as $x$, and $y$
belongs to distinct copy of $K_{t}$ than $x$, we conclude that $y$ and $s$ do
not belong to the same copy of $K_{t}$. Further, since $x\not \in S^{\prime}$
and $s\in S^{\prime}$, it follows that $x$ and $s$ are two distinct vertices.
So, from the fact that $x$ and $s$ belong to a same copy of $K_{t}$ we infer
that $x$ and $s$ belong to distinct copies of $K_{k}$. Since $y$ belongs to
the same copy of $K_{k}$ as $x$, this implies that $y$ and $s$ belong to two
distinct copies of $K_{k}$ too. We conclude that $y$ and $s$ are not connected
by an edge in $K_{t}\square K_{k}$. Hence, $x$ and $y$ are distinguished by
$s\in S^{\prime}$.

Assume finally that $x$ belongs to $V_{1}$ and $y$ to $V_{2}$. Again, we may
assume that $y\not \in S^{\prime}$, as otherwise $x$ and $y$ would be
distinguished by $y\in S^{\prime}$. Denote by $S^{\prime\prime}$ the set of
all vertices of $S^{\prime}$ which belong to the same copy of $K_{t}$ as $y$.
Notice that $\left\vert S^{\prime\prime}\right\vert \geq t-1-(k-1)=t-k$.
Observe that $y$ is connected by an edge in $K_{t}\square K_{k}$ to every
vertex of $S^{\prime\prime}$. If there exists a vertex $s$ in $S^{\prime
\prime}$ such that $x$ is not adjacent to $s$ in $K_{t}\square K_{k}$, then
$x$ and $y$ are distinguished by $s\in S^{\prime}$. So, let us assume that $x$
is adjacent to every vertex of $S^{\prime\prime}$. This implies that $x$ is
adjacent to at least $\left\vert S^{\prime\prime}\right\vert \cdot
k\geq(t-k)k$ vertices of $V_{2}$, at least $(t-k)k-(k-1)$ of which belong to
$S^{\prime}$. On the other hand, $y$ is adjacent to $(t-1)+(k-1)$ vertices of
$V_{2}$, and hence to at most $(t-1)+(k-1)$ vertices of $S^{\prime}$. Observe
that if $t>k+3$, then
\[
(t-k)k-(k-1)>(t-1)+(k-1).
\]
So if $t>k+3$, then there exists at least one vertex from $S^{\prime}$ which
is adjacent to $x$ and it is not adjacent to $y$. We conclude that $x$ and $y$
are distinguished by $S^{\prime}$ and the claim is established.

\medskip

Claim A implies that $\mathrm{dim}_{k}(M_{t,k})\leq tk$. In order to prove
that $\mathrm{dim}_{k}(M_{t,k})=tk$, we need the following claim.

\medskip

\noindent\textbf{Claim B.}\emph{ If }$S\subseteq V(M_{t,k})$\emph{ is a }%
$k$\emph{-resolving set in }$M_{t,k}$\emph{ and }$\left\vert S\cap
V_{2}\right\vert <tk$,\emph{ then }$\left\vert S\right\vert \geq
2^{t-1}.\smallskip$

\noindent Since $S$ is a $k$-resolving set of $M_{t,k}$, a set $S^{\prime
}=S\setminus\{s_{1},\ldots,s_{k-1}\}$ must be a resolving set in $M_{t,k}$ for
every subset $\{s_{1},\ldots,s_{k-1}\}\subseteq S$ of cardinality $k-1$. Since
$\left\vert S\cap V_{2}\right\vert <tk$, it follows that $S$ does not contain
at least one vertex of $V_{2}$. The vertex of $V_{2}$ not contained in $S$
belongs to a copy of $K_{k}$. So choose $\{s_{1},\ldots,s_{k-1}\}$ to contain
all the remaining vertices of $K_{k}$ in $S$. Then $S^{\prime}=S\setminus
\{s_{1},\ldots,s_{k-1}\}$ does not contain any vertex of this copy of $K_{k}$.
Assume that the copy of $K_{k}$ not contained in $S^{\prime}$ consists of
vertices $v_{i,j}$ for $1\leq j\leq k$. Notice that there are $2^{t-1}$ pairs
of vertices in $V_{1}$ which differ only in the $i$-th coordinate and all
these pairs are pairwise vertex disjoint. As these pairs are distinguished
only by vertices $v_{i,j}$ of $V_{2}$ for $1\leq j\leq k$, it follows that
they must be distinguished by vertices of $S^{\prime}\cap V_{1}$. Recall that
vertices of $V_{1}$ induce a complete subgraph of $M_{t,k}$, so each pair of
vertices in $V_{1}$ is distinguished only by the two vertices which belong to
the pair. Since there are $2^{t-1}$ pairs of $V_{1}$ to distinguish, and all
these pairs are pairwise vertex disjoint, it follows that $\left\vert
S^{\prime}\right\vert \geq2^{t-1}$. From $S^{\prime}\subseteq S$ we conclude
that $\left\vert S\right\vert \geq2^{t-1}$ also holds, and the claim is established.

\medskip

In Claim A we have established that $S=V_{2}$ is a $k$-resolving set of
$M_{t,k}$. Claim B implies that any other $k$-resolving set of $M_{t,k}$
contains at least $2^{t-1}$ vertices. Since $2^{t-1}\geq kt=\left\vert
V_{2}\right\vert $ for big enough $t$, we conclude that $S$ is a smallest
$k$-resolving set of $M_{t,k}$. Hence, $\mathrm{dim}_{k}(M_{t,k})=kt$.

It remains to prove that $\mathrm{dim}_{k+1}(M_{t,k})\geq2^{t-1}$. In order to
establish this, we need to show that every $(k+1)$-resolving set $S$ of $G$
contains at least $2^{t-1}$ vertices. Let $S$ be a $(k+1)$-resolving set of
$M_{t,k}$. This implies that $S$ is also a $k$-resolving set of $G$. If $S$
does not contain a vertex of $V_{2}$, then Claim B implies that $\left\vert
S\right\vert \geq2^{t-1}$. So, let us assume that $S$ contains all vertices of
$V_{2}$. Since $S$ is a $(k+1)$-resolving set, the set $S^{\prime}%
=S\setminus\{s\}$ must be a $k$-resolving set for every $s\in S$. So choose
$s\in S\cap V_{2}$. Then $S^{\prime}=S\setminus\{s\}$ is a $k$-resolving set
which does not contain a vertex of $V_{2}$. By Claim~B we have $\left\vert
S^{\prime}\right\vert \geq2^{t-1}$. Since $S^{\prime}\subseteq S$, it follows
that $\left\vert S\right\vert \geq2^{t-1}$, and we are done.
\end{proof}

\section{Concluding remarks}

In this paper we consider the tolerance of metric basis to faults. A metric
basis $S\subseteq V(G)$ distinguishes all pairs of vertices of $G,$ but it
will not distinguish them any more if one of its vertices is excluded from
$S.$ Since the concept of metric basis is used to model the sensors placed in
a network, and a sensor in a physical world can malfunction, it is of interest
to consider the problem of making metric dimension tolerant to faults. A
natural question is how many more vertices need to be added to a metric basis
in order to make it tolerant to the fault of single sensor. In other words,
the question is how large $\mathrm{ftdim}(G)$ can be in terms of
$\mathrm{dim}(G).$ There are classes of graphs for which $\mathrm{ftdim}(G)$
is linearly bounded by $\mathrm{dim}(G)$ as presented in Propositions
\ref{Prop_trees} and \ref{Prop_multipartite}, but a general known bound (see
Theorem \ref{Tm_spanci}) is
\[
\mathrm{ftdim}(G)\leq\mathrm{dim}(G)(1+2\cdot5^{\mathrm{dim}(G)-1})
\]
which is exponential in terms of $\mathrm{dim}(G).$ In this paper, we provide
a graph $G$, namely $M_{t},$ with
\[
\mathrm{ftdim}(G)\geq\mathrm{dim}(G)+2^{\mathrm{dim}(G)-1},
\]
and this for every possible value of $\mathrm{dim}(G)$ (see Theorem
\ref{Tm_ft}). By this, we show that the bound on $\mathrm{ftdim}(G)$ indeed
has to be exponential in terms of $\mathrm{dim}(G).$

\begin{figure}[h]
\begin{center}
\includegraphics[scale=0.6]{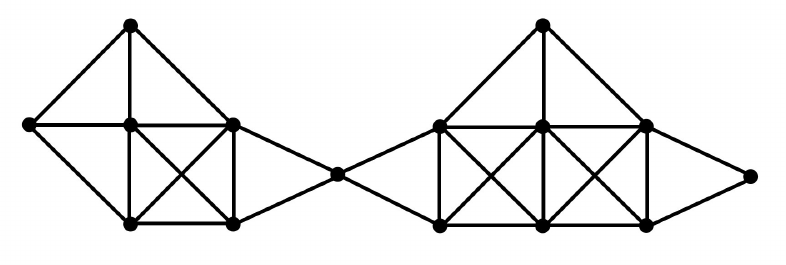}
\end{center}
\caption{A graph $G$ on $n=15$ vertices with $\mathrm{dim}(G)=2$ and
$\mathrm{ftdim}(G)=7.$}%
\label{Fig04}%
\end{figure}

On the other hand, the gap between the upper bound and the value of
$\mathrm{ftdim}(G)$ for the graph we provide is still significant. An
interesting direction for further work is to narrow this gap by either
lowering the upper bound or finding a graph with the value of $\mathrm{ftdim}%
(G)$ higher than for our graph. If we consider graphs with $\mathrm{dim}%
(G)=2,$ it is not difficult to find graphs with $\mathrm{ftdim}(G)=7$ which is
higher than $\mathrm{dim}(G)+2^{\mathrm{dim}(G)-1}=4.$ One such graph is shown
in Figure \ref{Fig04}. For higher values of $\mathrm{dim}(G),$ given the
nature of an exponential function, the gap between $\mathrm{dim}(G)$ and
$\mathrm{ftdim}(G)$ must increase significantly, so it seems more difficult to
find an example of a graph better than $M_{t}.$

For every integer $t\geq1,$ let us define $j(t)$ to be a maximum value of
$\mathrm{ftdim}(G)$ over all graphs $G$ with $\mathrm{dim}(G)=2$ and
$\kappa(G)\geq2,$ where we do not fix the order of the graph $G.$ Since only
paths $P_{n}$ have $\mathrm{dim}(P_{n})=2,$ and this for every $n\geq2,$ it
holds that $j(1)=2.$ For $t=2,$ notice that the upper and the lower bounds on
$\mathrm{ftdim}(G)$ imply $4\leq j(2)\leq22.$ In the light of the graph $G$
from Figure \ref{Fig04}, this can be narrowed to $7\leq j(2)\leq22.$

As for $t\geq3$, among graphs with $\mathrm{dim}(G)=t,$ the graph $M_{t}$ is
the graph with the highest value of $\mathrm{ftdim}(G)$ we know of. Therefore,
the best known bounds for $j(t)$ are provided by Theorems \ref{Tm_spanci} and
\ref{Tm_ft}, i.e.
\begin{equation}
t+2^{t-1}\leq j(t)\leq t(1+2\cdot5^{t-1}). \label{For_jt}%
\end{equation}
Motivated by all this, we propose the following problem.

\begin{problem}
Determine the function $j.$
\end{problem}

\noindent As the above problem can be challenging, one may find tractable the
following subtasks:

\begin{itemize}
\item \textit{find }$j(2)$;

\item \textit{for }$t\geq2$\textit{, narrow the interval of possible values of
}$j(t)$\textit{ given by }(\ref{For_jt}).
\end{itemize}

We also established the similar results for the $k$-metric dimension, which is
a generalization of the fault-tolerant metric dimension. First, we established
the following upper bound on the $(k+1)$-metric dimension
\[
\mathrm{dim}_{k+1}(G)\leq\mathrm{dim}_{k}(G)(1+2\cdot5^{\mathrm{dim}_{k}%
(G)-1})
\]
which is again exponential in terms of $\mathrm{dim}_{k}(G).$ Next, for any
$k\geq2$ and sufficiently large integer $t,$ we provided a graph $G=M_{t,k}$
with $\mathrm{dim}_{k}(G)=kt$ and $\mathrm{dim}_{k+1}(G)\geq2^{t-1},$ which is
a confirmation that the upper bound on $\mathrm{dim}_{k+1}(G)$ indeed has to
be exponential in terms of $\mathrm{dim}_{k}(G).$ Yet again, it would be
interesting to:

\begin{itemize}
\item \textit{narrow the gap between the established upper bound and the value
of }$\mathrm{dim}_{k+1}(M_{t,k})$.
\end{itemize}

\bigskip

\bigskip\noindent\textbf{Acknowledgments.}~~First author acknowledges partial
support by Slovak research grants VEGA 1/0567/22, VEGA 1/0069/23,
APVV--22--0005 and APVV-23-0076. Second author acknowledges partial support by
Project KK.01.1.1.02.0027, a project co-financed by the Croatian Government
and the European Union through the European Regional Development Fund - the
Competitiveness and Cohesion Operational Programme. All authors acknowledge
partial support by the Slovenian Research Agency ARRS program\ P1-0383 and
ARRS project J1-3002.

\end{document}